\documentclass[11pt]{article}
\usepackage{amsfonts}
\usepackage{amssymb}
\usepackage{amssymb,amsmath,amsfonts,mathrsfs,amsthm,epsfig,latexsym,color}
\usepackage{enumitem,geometry}
\usepackage{overpic}

\makeatletter
\newcommand{\subsectionruninhead}{\@startsection{subsection}{2}{0mm}
{-\baselineskip}{-0mm}{\bf\large}}
\newcommand{\subsubsectionruninhead}{\@startsection{subsubsection}{3}{0mm}
{-\baselineskip}{-0mm}{\bf\normalsize}}
\makeatother

\newtheorem*{theorem*}{Theorem}
\newtheorem*{proof*}{Proof}

\newtheorem*{proposition*}{Proposition}
\newtheorem*{corollary*}{Corollary}
\newtheorem*{claim*}{Claim}
\newtheorem{theorem}{Theorem}
\newtheorem{proposition}{Proposition}[section]

\newtheorem{lemma}[proposition]{Lemma}

\theoremstyle{definition}

\theoremstyle{remark}
\newtheorem{remark}[proposition]{Remark}

\numberwithin{equation}{section}

 \def\RR{{\mathbb R}} \def\SS{{\mathbb S}} \def\TT{{\mathbb T}}
   
 \def\ZZ{{\mathbb Z}}

\begin{document}
\title{A chain transitive accessible partially hyperbolic diffeomorphism}
\author{SHAOBO GAN and YI SHI}
\date{}
\maketitle
\begin{abstract}
In this paper, we construct a partially hyperbolic skew-product diffeomorphism $f$ on $\mathbb{T}^3$, such that $f$ is accessible and chain transitive, but not transitive.
\end{abstract}

\section{Introduction}

Let $M$ be a closed Riemannian manifold, and $f:M\rightarrow M$ be a diffeomorphism. We say $f$ is transitive, if for any two open sets $U,V\subset M$, there exists $n>0$, such that $f^n(U)\cap V\neq\emptyset$. Transitivity is a notion to describe the mixing property of the dynamics generated by $f$. The transitivity of $f$ is equivalent to there exists a point $x$ whose positive orbit $\{f^n(x):n>0\}$ is dense in $M$.

We call a point $x\in M$ is a non-wandering point of $f$, if for any neighborhood $U_x$ of $x$, there exists $n>0$, such that $f^n(U_x)\cap U_x\neq\emptyset$. The non-wandering set $\Omega(f)$ is the set of all non-wandering points of $f$. It is clear that a point is a non-wandering point, then its orbit has somekind recurrent property.

For two points $x,y\in M$, we say $y$ is chain attainable from $x$, if for any $\epsilon>0$, there exists a finite sequence $\{x_i\}_{i=0}^n$ with $x_0=x$ and $x_n=y$, such that $d(f(x_i),x_{i+1})<\epsilon$ for any $0\leq i\leq n-1$. A point $x\in M$ is called a chain recurrent point, if it is chain attainable from itself. The set of chain recurrent points is called chain recurrent set of $f$, denoted by ${\rm CR}(f)$. If every point is chain recurrent, we say $f$ is chain transitive.

It is clear that if a point is non-wandering, then it must be chain recurrent. Similarly, if $f$ is transitive, then it must be chain transitive. However, from the powerful chain connecting lemma \cite{BC}, there exists a residual set ${\cal R}\subset{\rm Diff}^1(M)$, such that for any $f\in{\cal R}$, we have $\Omega(f)={\rm CR}(f)$. Moreover, for the classical Anosov diffeomorphisms, we must have their non-wandering sets are equal to chain recurrent sets.

A diffeomorphism $f:M\rightarrow M$ is partially hyperbolic, if the tangent bundle $TM$ splits into three nontrivial $Df$-invariant bundles $TM=E^{ss}\oplus E^c\oplus E^{uu}$,
such that $Df|_{E^{ss}}$ is uniformly contracting, $Df|_{E^{uu}}$ is uniformly expanding, and $Df|_{E^c}$ lies between them:
\begin{displaymath}
\parallel Df|_{E^{ss}(x)}\parallel<\parallel Df^{-1}|_{E^c(f(x))}\parallel^{-1},\quad \parallel Df|_{E^c(x)}\parallel<\parallel Df^{-1}|_{E^{uu}(f(x))}\parallel^{-1},\quad\textrm{for all}\ x\in M.
\end{displaymath}
It is known that there are unique $f$-invariant foliations ${\cal W}^{ss}$ and ${\cal W}^{uu}$ tangent to $E^s$ and $E^u$ respectively.

An important geometric property of partially hyperbolic diffeomorphisms is accessibility. A partially hyperbolic diffeomorphism $f$ is accessible, if for any two pints $x,y\in M$, they can be joined by an arc
consisting of finitely many segments contained in the leaves of foliations ${\cal W}^{ss}$ and ${\cal W}^{uu}$. Accessibility plays a key role for proving the ergodicity of partially hyperbolic diffeomorphisms, see \cite{BW,HHU}. Moreover, it has been observed that most of partially hyperbolic diffeomorphisms are accessible \cite{BHHTU,DW,HHU}.

M. Brin \cite{B} has proved that for a partially hyperbolic diffeomorphism $f:M\rightarrow M$, if $f$ is accessible and $\Omega(f)=M$, then $f$ is transitive. See also \cite{ABD}. So it is natural to ask the following question: {\it if a partially hyperbolic diffeomorphism $f$ is accessible and ${\rm CR}(f)=M$, is $f$ transitive?} In this paper, we construct an example which gives a negative answer to this question. This implies Brin's result could not be generalized for the case where ${\rm CR}(f)=M$.

Let $A:\mathbb{T}^2\rightarrow \mathbb{T}^2$ be a hyperbolic automorphism over $\mathbb{T}^2$. We say $f:\mathbb{T}^3\rightarrow \mathbb{T}^3$ is a partially hyperbolic skew-product over $A$, if for every $(x,t)\in\mathbb{T}^3=\mathbb{T}^2\times\mathbb{S}^1$, we have
$$
f(x,t)=(Ax,\varphi_x(t)), \qquad {\rm and} \qquad
\|A^{-1}\|^{-1}<\|\varphi'_x(t)\|<\|A\|.
$$
We will consider $\SS^1=\RR/2\ZZ$, and usually use the coordinate $\SS^1=[-1,1]/\{-1,1\}$.

Our main result is the following theorem.

\begin{theorem}\label{Main}
There exists a partially hyperbolic skew-product $C^{\infty}$-diffeomorphism $f:\mathbb{T}^3\rightarrow \mathbb{T}^3$, such that $f$ is accessible and chain transitive, but not transitive.
\end{theorem}

\begin{remark}
We want to point out that for $C^1$-generic diffeomorphisms, chain transitivity implies transitivity. Our construction need the help of nonhyperbolic periodic points. So we don't know for $C^r$-generic or $C^r$-open dense accessible partially hyperbolic diffeomorphisms, whether chain transitivity implies transitivity.
\end{remark}

\section{Construction of diffeomorphism}

The idea of our example is first we construct a chain transitive partially hyperbolic skew-product diffeomorphism on $\TT^3$, such that its non-wandering set is not the whole $\TT^3$ and not  transitive. Then we make a small perturbation to achieve the accessibility, and still preserving the dynamical properties.

First we need a diffeomorphism on $\SS^1$ that is chain transitive but the non-wandering set is not the whole circle.

Let $\theta:\SS^1\rightarrow\SS^1$ be defined as
$$
\theta(t)=-\cos(2\pi t)+1, \qquad t\in\RR/2\ZZ.
$$
It is a $C^{\infty}$-smooth function on $\SS^1$. We can see that $\theta\geq0$ on $\SS^1$, and has two zero points $0$ and $-1=1$. The vector field $\{\theta(t)\cdot\frac{\partial}{\partial t}\}$ is a smooth vector field on $\SS^1$, and its time-$r$ map for $0<r\ll1$ is the diffeomorphism we need on the circle. See Figure 1. That is the time-$r$ map of $\theta(t)\cdot\frac{\partial}{\partial t}$ is chain transitive, and the non-wandering set consists of only two fixed points $0$ and $-1=1$. Using the product structure, we can define it on $\TT^3=\TT^2\times\SS^1$.

\begin{lemma}
The vector field $X$ defined by
$$
X(x,t)=\theta(t)\cdot\frac{\partial}{\partial t}, \qquad \forall (x,t)\in\TT^3=\TT^2\times\SS^1,
$$
is a smooth vector field on $\TT^3$. Moveover, for every $r>0$, the time-$r$ map $X_r$ of the flow generated by $X$ satisfies the following properties:
\begin{itemize}
  \item $X_r(x,t)=(x,\varphi_x(t))$ for every $(x,t)\in\TT^3$.
  \item For $i=0,1$, $X_r(x,i)=(x,i)$ for every $x\in\TT^2$.
  \item Fix $0<\delta\ll1$, then for every $(x,t)\in\TT^2\times([-\delta,0)\cup[1-\delta,1))$, we have
      $X_r(x,t)=(x,\varphi_x(t))$ satisfies $\varphi_x(t)>t$. In particularly, if we choose $r$ small enough, there exists $0<\tau=\tau(r,\delta)<\delta/2$, such that
      $$
      \varphi_x(t)>t+\tau, \qquad \forall (x,t)\in \TT^2\times\{-\delta,1-\delta\}.
      $$
\end{itemize}
\end{lemma}

\begin{figure}[htbp]
\centering
\includegraphics[width=15cm]{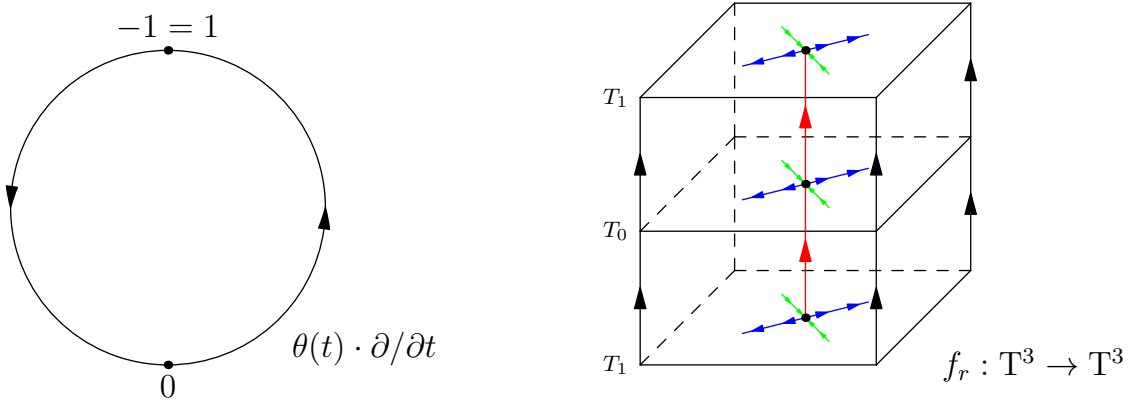}
\caption{Chain transitive systems with nonempty wandering sets.}
\end{figure}

For $r>0$ is small enough, we define the diffeomorphism $f_r=X_r\circ(A\times{\rm id}):\TT^3\rightarrow\TT^3$. Then with the same constants $\delta=\delta(r)$ and $\tau=\tau(r)$ in the last lemma, $f_r$ satisfies the following properties(Figuer 1):
\begin{itemize}
  \item $f_r$ is a partially hyperbolic skew-product diffeomorphism on $\TT^3$, $f_r(x,t)=(Ax,\varphi_{Ax}(t))$, where $\varphi_x(t)$ is exactly the same as $X_r$. Let the partially hyperbolic splitting be:
  $$
  T\TT^3=E^{ss}\oplus E^c\oplus E^{uu},
  $$
  and denote by $W^{ss/uu}$ the stable/unstable manifolds generated by $E^{ss/uu}$.
  \item In the fixed center fiber $S_p$, $f_r|_{S_p}$ is chain transitive and has two fix points $P_i=(p,i)\in\TT^2\times\SS^1$ for $i=0,1$.
  \item For $i=0,1$, $f_r$ preserves $\TT_i=\TT^2\times\{i\}$ invariant, and $f_r|_{\TT_i}=A|_{\TT_i}$. Moreover,
        $$
        \TT_i=\overline{W^{ss}(P_i,f_r)}=\overline{W^{uu}(P_i,f_r)}.
        $$
  \item For every $(x,t)\in \TT^2\times\{-\delta,1-\delta\}$, we have $\varphi_{Ax}(t)>t+\tau$.
\end{itemize}

\begin{remark}
We want to point out that $X_r$ and $A\times{\rm id}$ are commutable, thus $f_r=X_r\circ(A\times{\rm id})=(A\times{\rm id})\circ X_r$.
\end{remark}

Now $f_r$ is a chain transitive but nontransitive partially hyperbolic diffeomorphism on $\TT^3$. However, $f_r$ is not accessible, since the union of stable and unstable bundles of $f_r$ is integrable. We will make some more perturbations to achieve the accessibility, and preserving other dynamical properties.

Let $p\in\TT^2$ be the fixed point of the linear Anosov automorphism $A$. Take a small local chart $(U(p);(x_s,x_u))$ centered at $p$ in $\TT^2$, such that
$$
A(x_s,x_u)=(\lambda\cdot x_s,\lambda^{-1}\cdot x_u),
$$
for every $(x_s,x_u)\in[-10,10]_s\times[-10,10]_u\subset U(p)$. Here $\lambda$ is the eigenvalue of $A$ with $0<|\lambda|<1$, and we assume $1<\lambda^{-1}<10$ for the simplicity of symbols. In the rest of this paper, the local coordinate of $(U(p);(x_s,x_u))$ is the only coordinate we used in $\TT^2$, and we use it in $\TT^2$ without ambiguity.

\begin{remark}\label{Rem:small}
We want to point out that here we require the neighborhood $U(p)$ to be chosen very small, such that for any point $(0,x_u)$ with $x_u\neq 0$, there exists some $n>0$, such that $A^n(0,x_u)\notin U(p)$. The same holds for $(x_s,0)$ with $x_s\neq 0$, and its negative iterations of $A$.
\end{remark}

Now we define a $C^{\infty}$-smooth function $\alpha:\TT^2\rightarrow[0,1]$, such that
\begin{displaymath}
\alpha(x)=\left\{\begin{array}{ll}
                 0,&x\in[-1,1]_s\times[-1,1]_u\subset U(p),\\
                 1,&x\in\TT^2\setminus[-3,3]_s\times[-3,3]_u,\\
                 \in(0,1),& {\rm otherwise}.
                 \end{array}\right.
\end{displaymath}

The function $\alpha$ will help us to prescribe the perturbation region. And the next function $\gamma$ is used to show the way of perturbations.

Let $\gamma:\SS^1=[-1,1]/\{-1=1\}\rightarrow\RR$ be a $C^{\infty}$-smooth function, such that
\begin{displaymath}
\gamma(t):\left\{\begin{array}{ll}
                 >0,&t\in[-1,-1+\tau)\cup(-\tau,\tau)\cup(1-\tau,1],\\
                 =0,&t\in[-1+\tau,-\tau]\cup[\tau,1-\tau].
                 \end{array}\right.
\end{displaymath}
We define a smooth vector field $Y$ on $\TT^3$ by
$$
Y(x,t)=-\alpha(x)\gamma(t)\cdot\frac{\partial}{\partial t}, \qquad \forall (x,t)\in\TT^3=\TT^2\times\SS^1.
$$

\begin{figure}[htbp]
\centering
\includegraphics[width=9cm]{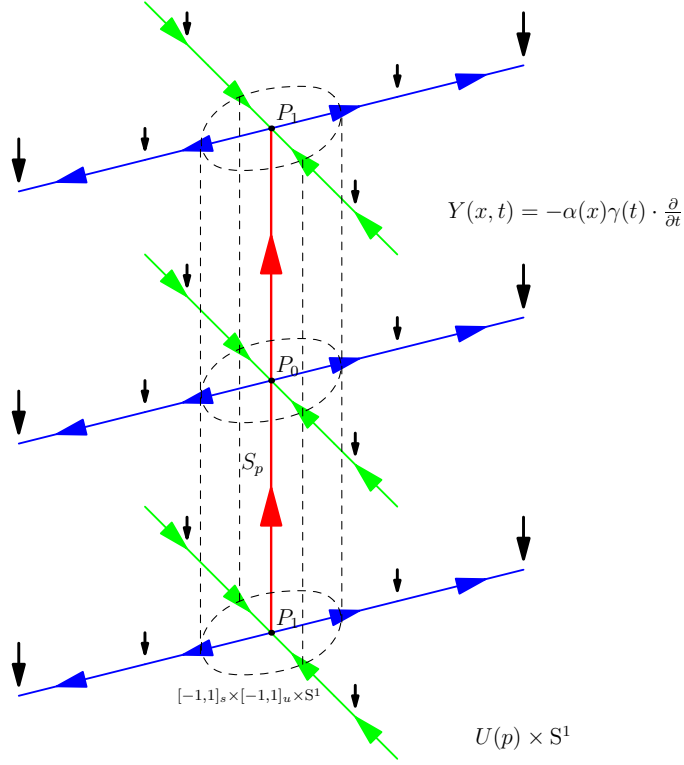}
\caption{The perturbation made by $Y_{\rho}$.}
\end{figure}

Recall that $\tau<\delta/2$, and so for $\rho>0$ small enough, the time-$\rho$ map $Y_{\rho}$ satisfies the following properties(see Figure 2):
\begin{itemize}
  \item $Y_{\rho}(x,t)=(x,\psi_x(t))$, and $Y_{\rho}(x,t)=(x,t)$ for every $(x,t)\in[-1,1]_s\times[-1,1]_u\times\SS^1$.
  \item For $i=0,1$, $\psi_x(i)\leq i$ for every $x\in\TT^2$. More precisely, for $i=0,1$,
      \begin{itemize}
      \item $\psi_x(i)=i$, for every $x\in[-1,1]_s\times[-1,1]_u$;
      \item $\psi_x(i)<i$, for every $x\in\TT^2\setminus[-1,1]_s\times[-1,1]_u$.
      \end{itemize}
  \item For every $(x,t)\in\TT^2\times([-1+\tau,-\tau]\cup[\tau,1-\tau])$, we have $Y_{\rho}(x,t)=(x,t)$. In particularly, for every $x\in\TT^2$ and $t\in\{-\delta,1-\delta\}$, $\psi_x(t)=t$.
\end{itemize}

Now we can considering the perturbation of $f_r$ made by $Y_{\rho}$, and it is the diffeomorphism we promised in our main theorem.

\begin{proposition}\label{Prop:Property}
The diffeomorphism $f=Y_{\rho}\circ f_r:\TT^3\rightarrow\TT^3$ satisfies the following properties:
\begin{enumerate}
  \item $f$ is a partially hyperbolic skew-product diffeomorphism:
      $$
      f(x,t)=(Ax,\psi_{Ax}\circ\varphi_{Ax}(t)), \qquad \forall (x,t)\in\TT^3.
      $$
  \item When restricted in the fixed fiber $S_p$, $f|_{S_p}$ has two fixed points $P_0,P_1$, and is chain transitive.
  \item For $i=0,1$, $\psi_{Ax}\circ\varphi_{Ax}(i)\leq i$ for every $x\in\TT^2$. More precisely, for $i=0,1$,
      \begin{itemize}
      \item
      $\psi_{Ax}\circ\varphi_{Ax}(i)=i$, for every $x\in[-\lambda^{-1},\lambda^{-1}]_s\times[-\lambda,\lambda]_u$.
      \item
      $\psi_{Ax}\circ\varphi_{Ax}(i)< i$, for every $x\in\TT^2\setminus[-\lambda^{-1},\lambda^{-1}]_s\times[-\lambda,\lambda]_u$.
      \end{itemize}
  \item For $t\in\{-\delta,1-\delta\}$, $\psi_{Ax}\circ\varphi_{Ax}(t)>t+\tau$ for every $x\in\TT^2$.
\end{enumerate}
\end{proposition}

\begin{proof}
the first item comes from the skew-product structure of $Y_{\rho}$ and $f_r$. The second item from the vector field $Y$ vanishes in a neighborhood of $S_p$. The third item comes from the fact that $f_r$ preserves two tori $\TT_0$ and $\TT_1$ invariant, and the second property of $Y_{\rho}$. The last item holds because $\psi_{Ax}(t)=t$ for every $x\in\TT^2$ and $t\in\{-\delta,1-\delta\}$.
\end{proof}

\section{Dynamical and geometrical properties of $f$}

Now we can proof the main theorem from the following three lemmas.

\begin{lemma}\label{Lem:Chain-Transitive}
The diffeomorphism $f:\TT^3\rightarrow\TT^3$ is chain transitive.
\end{lemma}

\begin{proof}
From the first and second properties of $f$ in Proposition \ref{Prop:Property}, we know that $f$ is a partially hyperbolic skew-product diffeomorphism on $\TT^3$, thus the stable and unstable manifolds of the fixed fiber $S_p$ are dense on $\TT^3$. Since $f|_{S_p}$ is chain transitive, this implies $f$ is chain transitive on $\TT^3$.
\end{proof}

\begin{lemma}\label{Lem:Accessible}
The diffeomorphism $f:\TT^3\rightarrow\TT^3$ is accessible.
\end{lemma}

\begin{proof}
Since $f$ is a partially hyperbolic skew-product diffeomorphism on $\TT^3$, if $f$ is not accessible, then from theorem 1.6 of \cite{H}, $f$ has a compact $us$-leaf. Here $us$-leaf is a compact complete 2-dimensional submanifold which is tangent to $E^{ss}\oplus E^{uu}$ of $f$. It is a torus transverse to the $\SS^1$-fiber of $\TT^3$. Since the compact $us$-leaf is saturated by ${\cal W}^{ss}$ and ${\cal W}^{uu}$, it intersects every $\SS^1$-fiber of $\TT^3$. Moreover, this $us$-leaf must intersect every $\SS^1$-fiber with only finitely many points, which comes from it is a compact and complete submanifold.

If this compact $us$-leaf is not periodic by $f$, then theorem 1.9 of \cite{H} shows that $f$ is semi-conjugated to $A$ times an irrational rotation on $\SS^1$, this implies $f$ has no periodic points. This contradicts to $P_0$ and $P_1$ are two fixed points of $f$, thus $f$ must have a periodic compact $us$-leaf $\TT_{us}$.

From the periodicity of $\TT_{us}$, we know that $\TT_{us}\cap S_p$ only contains $P_0$ or $P_1$, and $f(\TT_{us})=\TT_{us}$. Assuming $P_0\in\TT_{us}$, then from Theorem 1.7 of \cite{H}, we have
$$
\TT_{us}=\overline{W^{ss}(P_0,f)}=\overline{W^{uu}(P_0,f)}.
$$
In particularly, $W^{ss}(P_0,f)$ and $W^{uu}(P_0,f)$ has strong homoclinic intersections.

Recall that from the construction of $f$,
$$
W^{uu}(P_0,f)\cap\{0_s\}\times[-1,1]_u\times\SS^1=\{0_s\}\times[-1,1]_u\times\{0\}.
$$
Since $W^{uu}(P_0,f)=\cup_{n>0}f^n(W^{uu}_{loc}(P_0,f))$ and $U(p)$ is very small(remark \ref{Rem:small}), property 3 of Proposition \ref{Prop:Property} implies for every $(x,t)\in W^{uu}(P_0,f)\setminus\{0_s\}\times[-1,1]_u\times\SS^1$, we have $t<0$. On the other hand, from property 4 of Proposition \ref{Prop:Property},  for every $(x,t)\in W^{uu}(P_0,f)$, we know that $t>-\delta+\tau$ and hence $-\delta+\tau<t\le 0$.

However, we know that
$$
W^{ss}(P_0,f)\cap[-1,1]_s\times\{0_u\}\times\SS^1=[-1,1]_s\times\{0_u\}\times\{0\},
$$
and $W^{ss}(P_0,f)=\cup_{n>0}f^{-n}(W^{ss}_{loc}(P_0,f))$. From the construction of $f$, for every $(x,t)\in W^{ss}(P_0,f)$, we have $0\leq t<1-\delta$. This implies
$$
W^{ss}(P_0,f)\cap W^{uu}(P_0,f)=\{P_0\},
$$
which is a contradiction. The same argument works for $P_1\in\TT_{us}$, thus $f$ must be accessible.
\end{proof}

\begin{lemma}\label{Lem:Non-Transitive}
The diffeomorphism $f:\TT^3\rightarrow\TT^3$ is not transitive.
\end{lemma}

\begin{proof}
From the proof of last lemma, we know that $W^{uu}(P_0,f)\subset\TT^2\times[-\delta,0]$ and $W^{uu}(P_1,f)\subset\TT^2\times[1-\delta,1]$. So we can define two disjoint compact $f$-invariant $u$-saturated sets
$$
\Lambda_i=\overline{W^{uu}(P_i,f)}\subset\TT^2\times[i-\delta,i],\qquad i=0,1.
$$
Notice that $\Lambda_i$ intersects every center leaf.

Now we choose two open sets $U\subset\TT^2\times(-1,-\delta)$ and $V\subset\TT^2\times(0,1-\delta)$. Then we must have $f^n(U)\cap V=\emptyset$ for every $n>0$. Otherwise, there exists $Q=(q,t)\in U$ and $f^k(Q)=(A^kq,t')\in V$ for some $k>0$. Moreover, there exists some point $R=(A^kq,t'+s_1)\in\Lambda_1$, such that $0<s_1<1-\delta$, and the center interval $[f^k(Q),R)$ does not intersect $\Lambda_0$.

However, there exists some $0<s_0<1-t$, such that the point $Q'=(q,t+s_0)\in\Lambda_0$, and $\{q\}\times[t,s_0)\cap\Lambda_1=\emptyset$. From the invariance of $\Lambda_0$ and $\Lambda_1$, and $f$ preserves the orientation of $\SS^1$-fiber, for every $n>0$, the center curve started from $f^n(Q)$ will meet $f^n(Q')\in\Lambda_0$, and the center interval $[f^n(Q),f^n(Q'))$ does not intersect $\Lambda_1$. This is a contradiction for $n=k$. This proves $f$ is not transitive.
\end{proof}

\begin{remark}
Actually, the same idea we can prove the following generalized statement. Let $f:\TT^3\rightarrow\TT^3$ be a partially hyperbolic skew-product diffeomorphism. If $f$ preserves the orientation of center foliation, and has two disjoint invariant compact $u$-saturated sets, then $f$ is not transitive. In particularly, if $f$ is transitive, then it has only one minimal $u$-saturated. In a similar spirit, \cite{HP} shows that every partially hyperbolic diffeomorphism on the nonabelian 3-nilmanifolds has only one minimal $u$-saturated set.
\end{remark}

\noindent Shaobo Gan, School of Mathematical Sciences, Peking University, Beijing 100871, China\\
E-mail address: gansb@pku.edu.cn
\vspace{0.5cm}

\noindent Yi Shi, School of Mathematical Sciences, Peking University, Beijing 100871, China\\
E-mail address: shiyi@math.pku.edu.cn

\end{document}